\documentclass[11pt]{article}
\usepackage{amsmath,amssymb}

\newtheorem{propo}{{\bf Proposition}}[section]
\newtheorem{coro}[propo]{{\bf Corollary}}
\newtheorem{lemma}[propo]{{\bf Lemma}} \newtheorem{theor}[propo]{{\bf
Theorem}} \newtheorem{ex}{{\sc Example}}[section]

\newenvironment{proof}{{\bf Proof.}}{$\Box$}
\def\R{{\mathbb R}}

\def\N{{\mathbb N}}

\begin{document}

\vspace*{1.0in}

\begin{center} SUBALGEBRAS THAT COVER OR AVOID CHIEF FACTORS OF LIE ALGEBRAS  
\end{center}
\bigskip

\begin{center} DAVID A. TOWERS 
\end{center}
\bigskip

\begin{center} Department of Mathematics and Statistics

Lancaster University

Lancaster LA1 4YF

England

d.towers@lancaster.ac.uk 
\end{center}
\bigskip

\begin{abstract} We call a subalgebra $U$ of a Lie algebra $L$ a $CAP$-subalgebra of $L$ if for any chief factor $H/K$
of $L$, we have $H \cap U = K \cap U$ or $H+U = K+U$. In this paper we investigate some properties of such subalgebras and obtain some characterizations for a finite-dimensional Lie algebra $L$ to be solvable under the assumption that some of its maximal subalgebras or $2$-maximal subalgebras be $CAP$-subalgebras.
\medskip 

\end{abstract}

\section{The covering and avoidance property}
Throughout, $L$ will denote a finite-dimensional Lie algebra over a field $F$.
Let 
\begin{equation} 0 = A_0 \subset A_1 \subset \ldots \subset A_n = L  
\end{equation} be a chief series for $L$. The subalgebra $U$ {\em avoids} the factor algebra $A_i/A_{i-1}$ if $U \cap A_i = U \cap A_{i-1}$; likewise, $U$ {\em covers} $A_i/A_{i-1}$ if $U + A_i = U + A_{i-1}$. We say that $U$ has the covering and avoidance property of $L$ if $U$ either covers or avoids every chief factor of $L$. We also say that $U$ is a $CAP$-subalgebra of $L$. The corresponding concepts in group theory have been studied extensively and have proved useful in characterising finite solvable groups and some of their subgroups (see, for example, \cite{tom}, \cite{xs} and \cite{ld}). In Lie algebras, some parallel results have been obtained by a number of authors, and this paper is intended to be a further contribution to that work.

There are a number of ways in which $CAP$-subalgebras arise. We say that $A_i/A_{i-1}$ is a {\em Frattini} chief factor if $A_i/A_{i-1} \subseteq \phi(L/A_{i-1})$; it is {\em complemented} if there is a maximal subalgebra $M$ of $L$ such that $L = A_i + M$ and $A_i \cap M = A_{i-1}$. When $L$ is solvable it is easy to see that a chief factor is Frattini if and only if it is not complemented. For a subalgebra $B$ of $L$ we denote by $[B:L]$ the set of all subalgebras $S$ of $L$ with $B \subseteq S \subseteq L$, and by $[B:L]_{max}$ the set of maximal subalgebras in $[B:L]$; that is, the set of maximal subalgebras of $L$ containing $B$. We define the set $\mathcal{I}$ by $i \in \mathcal{I}$ if and only if $A_i/A_{i-1}$ is not a Frattini chief factor of $L$. For each $i \in \mathcal{I}$ put
\[ \mathcal{M}_i = \{ M \in [A_{i-1}, L]_{max} \colon A_i \not \subseteq M\}.
\]
Then $U$ is a {\em prefrattini} subalgebra of $L$ if 
\[ U = \bigcap_{i \in \mathcal{I}} M_i \hbox{ for some } M_i \in \mathcal{M}_i.
\]
It was shown in \cite{prefrat} that, when $L$ is solvable, this definition does not depend on the choice of chief series, and  that the prefrattini subalgebras of $L$ cover the Frattini chief factors and avoid the rest; that is, they are $CAP$-subalgebras of $L$.
\par

Further examples were given by Stitzinger in \cite{stit}, where he proved the following result (see \cite{stit} for definitions of the terminology used).

\begin{theor}(\cite[Theorem 2]{stit} Let ${\mathcal F}$ be a saturated formation of solvable Lie algebras, and let $U$ be an ${\mathcal F}$-normaliser of $L$. Then $U$ covers every ${\mathcal F}$-central chief factor of $L$ and avoids every ${\mathcal F}$-eccentric chief factor of $L$.
\end{theor}
\bigskip

The chief factor $A_i/A_{i-1}$ is called {\em central} if $[L,A_i] \subseteq A_{i-1}$ and {\em eccentric} otherwise. A particular case of the above result is the following theorem, due to Hallahan and Overbeck.

\begin{theor}\label{t:ho} (\cite[Theorem 1]{ho}) Let $L$ be a metanilpotent Lie algebra. Then $C$ is a Cartan subalgebra of $L$ if and only if it covers the central chief factors and avoids the eccentric ones.
\end{theor}
\bigskip

In group theory an important class of $CAP$-subgroups is given by the {\em normally embedded} (also called {\em strongly pronormal}) subgroups (see \cite[page 251]{dh}). In a sense, the natural analogue of this concept in Lie algebras is to call a subalgebra $U$ of $L$ {\em strongly pronormal} if every Cartan subalgebra of $U$ is also a Cartan subalgebra of $U^L$, the ideal closure of $U$ in $L$. Such subalgebras satisfy a number of the same properties as those of their group-theoretic counterparts. However, they are not necessarily $CAP$-subalgebras, even when $L$ is metabelian, as the following example shows. 
\begin{ex} Let $L$ be the four-dimensional real Lie algebra with basis $e_1$, $e_2$, $e_3$, $e_4$ and multiplication $[e_1,e_3]=e_1$, $[e_2,e_3]=e_2$, $[e_1,e_4]=-e_2$ and $[e_2,e_4]=e_1$, other products being zero. Then $A = \R e_1+ \R e_2$ is a minimal abelian ideal of $L$ and $U = \R e_1+ \R e_3$ is strongly pronormal in $L$ (since the Cartan subalgebras of $U$ are of the form $\R (\alpha e_1 + e_3)$ ($\alpha \in \R$) and these are also Cartan subalgebras of $U^{L} = \R e_1 + \R e_2 + \R e_3$). However, $U \cap A = \R e_1 \neq U \cap 0$ and $U+A = \R e_1 + \R e_2 + \R e_3 \neq U+0$, so $U$ is not a $CAP$-subalgebra of $L$.
\end{ex}
\bigskip

An alternative approach which does yield examples of $CAP$-subalgebras will be given in the next section.

\section{Elementary results}
In this section we collect together some properties of $CAP$-subalgebras and give characterisations of simple and of supersolvable Lie algebras in terms of them. If $U$ is a subalgebra of $L$, the {\em core} of $U$, denoted $U_L$, is the largest ideal of $L$ contained in $U$.

\begin{lemma}\label{l:pre} Let $B$ be a subalgebra of $L$ and $H/K$ a chief factor of $L$. Then
\begin{itemize}
\item[(i)] $B$ covers $H/K$ if and only if $B \cap H +K = H$; and
\item[(ii)] $B$ avoids $H/K$ if and only if $(K + B) \cap H = K$.
\item[(iii)] If $B \cap H+K$ is an ideal of $L$, then $B$ covers or avoids $H/K$. In particular, ideals are $CAP$-subalgebras.
\item[(iv)] The non-trivial Lie algebra $L$ is simple if and only if it has no non-trivial proper $CAP$-subalgebras.
\item[(v)] $B$ covers or avoids $H/K$ if and only if there exists an ideal $N$ with $N \subseteq B \cap K$ and $B/N$ covers or avoids $(H/N)/(K/N)$ respectively. Furthermore, $B$ is a $CAP$-subalgebra of $L$ if and only if there exists an ideal $N$ of $L$ such that $N \subseteq B$ and $B/N$ is a $CAP$-subalgebra of $L/N$.
\item[(vi)] Let $C$ be a subalgebra containing $B$. If $H/K$ is covered (respectively, avoided) by $B$, then so is $(H \cap C)/(K \cap C)$.
\end{itemize}
\end{lemma}
\begin{proof} (i), (ii) These are straightforward.
\par

\noindent (iii) Since $B \cap H+K$ is an ideal of $L$, we have that  $B\cap H+K = H$ or  $B\cap H+K = K$. The former implies that $B$ covers $H/K$, by (i); the latter yields that $(K + B) \cap H =(B\cap H)+K = K$, whence $B$ avoids $H/K$, by (ii).
\par

\noindent (iv) This is straightforward.
\par

\noindent (v) Let $N = (B \cap K)_L$. Then
$$ B+H = B+K \Leftrightarrow \frac{B}{N} + \frac{H}{N} = \frac{B}{N} + \frac{K}{N}, \hbox{ and}$$
$$B \cap H = B \cap K \Leftrightarrow \frac{B}{N} \cap \frac{H}{N} = \frac{B}{N} \cap \frac{K}{N}. $$
\par

\noindent (vi) This is straightforward.
\end{proof}
\bigskip

A subalgebra $U$ of $L$ will be called {\em ideally embedded} in $L$ if $I_L(U)$ contains a Cartan subalgebra of $L$, where $I_L(U) = \{ x \in L : [x,U] \subseteq U \}$ is the {\em idealiser} of $U$ in $L$ . Clearly, any subalgebra containing a Cartan subalgebra of $L$ and any ideal of $L$ is ideally embedded in $L$. Then we have the following extension of Theorem \ref{t:ho}.

\begin{theor}\label{t:ie} Let $L$ be a metanilpotent Lie algebra and let $U$ be ideally embedded in $L$. Then $U$ is a $CAP$-subalgebra of $L$.
\end{theor}
\begin{proof} Let $C \subseteq I_L(U)$ be a Cartan subalgebra of $L$ and let $N$ be the nilradical of $L$. Then $(C+N)/N$ is a Cartan subalgebra of $L$ and $L/N$ is nilpotent, so $L=C+N$. Let $H/K$ be a chief factor of $L$. Then $[N,H]\subseteq K$ so $U\cap H+K$ is an ideal of $L$. The result now follows from Lemma \ref{l:pre} (iii)
\end{proof}
\bigskip

We define the {\em nilpotent residual}, $\gamma_{\infty}(L)$, of $L$ to be the smallest ideal of $L$ such that $L/\gamma_{\infty}(L)$ is nilpotent. Clearly this is the intersection of the terms of the lower central series for $L$. Then the {\em lower nilpotent series} for $L$ is the sequence of ideals $N_i(L)$ of $L$ defined by $N_0(L) = L$, $N_{i+1}(L) = \gamma_{\infty}(N_i(L))$ for $i \geq 0$. Then we have the following extension of Theorem \ref{t:ie}.

\begin{coro}\label{c:iesolv} Let $L$ be any solvable Lie algebra and let $U$ be an ideally embedded subalgebra of $L$ with $K = N_2(L) \subseteq U$. Then $U$ is a $CAP$-subalgebra of $L$.
\end{coro}
\begin{proof} Let $C \subseteq I_L(U)$ be a Cartan subalgebra of $L$. Then $(C+K)/K$ is a Cartan subalgebra of $L/K$, and $I_{L/K}(U/K) \supseteq (I_L(U)+K)/K \supseteq (C+K)/K$, so $U/K$ is ideally embedded in $L/K$. Moreover, $L/K$ is metanilpotent. It follows from Theorem \ref{t:ie} that $U/K$ is a $CAP$-subalgebra of $L/K$. But now Lemma \ref{l:pre} (v) yields that $U$ is a $CAP$-subalgebra of $L$.
\end{proof}
\bigskip

 Let $U$ be a subalgebra of $L$ and $B$ an ideal of $L$. Then $U$ is said to be a {\em supplement} to $B$ in $L$ if $L=U+B$. Another set of examples of $CAP$-subalgebras, which don't require $L$ to be solvable, is given by the next result.

\begin{theor}Let $L$ be any Lie algebra, let $U$ be a supplement to an ideal $B$ in $L$, and suppose that $B^k \subseteq U$ for some $k \in \N$. Then $U$ is a $CAP$-subalgebra of $L$.
\end{theor}
\begin{proof} Let $L = B + U$ and let $H/K$ be a chief factor of $L$. Then  $K + [B,H] = H$ or $K$. Suppose first that $K + [B,H] = H$. Then $[B,H] \subseteq K + [B,[B,H]]$ and a simple induction argument shows that $H \subseteq K + B^k$ for all $k \geq 1$. Hence $H \subseteq K+U$, which yields that $H+U=K+U$.

So suppose now that $K + [B,H] = K$, whence $[B,H] \subseteq K$. Then $K + U \cap H$ is an ideal of $L$, and the result now follows from Lemma \ref{l:pre} (iii).
\end{proof}

\begin{lemma}\label{l:ideal} Let $U$ be a $CAP$-subalgebra of $L$ and let $B$ be an ideal of $L$. Then $B+U$ is a $CAP$-subalgebra of $L$.
\end{lemma}
\begin{proof} Let $H/K$ be a chief factor of $L$. If $U+H=U+K$ then $B+U+H=B+U+K$, so suppose that $U \cap H = U \cap K$. Similarly, since $B$ is a $CAP$-subalgebra, by Lemma \ref{l:pre} (iii), we can suppose that $B \cap H = B \cap K$. Then
$$ \frac{B+H}{B+K} \cong \frac{(B+H)/B}{(B+K)/B} \cong \frac{H/B \cap H}{K/B \cap K} \cong \frac{H}{K}$$
is a chief factor of $L$. If $U+B+H=U+B+K$ the result is clear, so suppose that $U \cap (B+H) = U \cap (B+K)$.
\par

Let $x \in (B+U) \cap H$. Then $x = b+u$ for some $b \in B$, $u \in U$, and $x \in H$. It follows that $u \in (B+H) \cap U = (B+K) \cap U$, so that $x \in (B+K) \cap H = K + B \cap H = K$. Thus $(B+U) \cap H \subseteq (B+U) \cap K$. But the reverse inclusion is clear and the result follows.
\end{proof}
\bigskip

The next result gives the dimension of $CAP$-subalgebras in terms of the chief factors that they cover.

\begin{lemma} Let $U$ be a $CAP$-subalgebra of $L$, let $ 0=A_0 < A_1 < \ldots < A_n = L$ be a chief series for $L$ and let ${\mathcal I} = \{i : 1 \leq i \leq n,  U \hbox{ covers } A_i/A_{i-1} \}$. Then $ \dim U = \sum_{i \in {\mathcal I}} (\dim A_i - \dim A_{i-1})$.
\end{lemma}
\begin{proof} We use induction on $n$. The result is clear if $n=1$. So suppose it holds for all Lie algebras with chief series of length $< n$, and let $L$ have a chief series of length $n$. Then $U+A_1/A_1$ is a $CAP$-subalgebra of $L/A_1$, by Lemmas \ref{l:ideal} and \ref{l:pre} (v). Moreover, 
$$ \dim (U+A_1/A_1) = \sum_{i \in {\mathcal I}, i \neq 1} (\dim A_i - \dim A_{i-1}),$$
by the inductive hypothesis. If $U$ covers $A_1/A_0$, then 
$$ \dim U = \dim (U+A_1) =  \dim (U+A_1/A_1) + \dim A_1 = \sum_{i \in {\mathcal I}} (\dim A_i - \dim A_{i-1}).$$
If $U$ avoids $A_1/A_0$ then
$$ \dim U = \dim (U/U \cap A_1) =  \dim (U+A_1/A_1) = \sum_{i \in {\mathcal I}} (\dim A_i - \dim A_{i-1}).$$ 
\end{proof}
\bigskip

Finally in this section we consider supersolvable Lie algebras, that is, Lie algebras all of whose chief factors are one-dimensional.

\begin{propo} Let $H/K$ be a chief factor of $L$. Then every one-dimensional subalgebra of $L$ covers or avoids $H/K$ if and only if $\dim (H/K) = 1$.
\end{propo}
\begin{proof} If $x \in K$ then $Fx = Fx \cap H = Fx \cap K$, so $Fx$ avoids $H/K$. If $x \notin H$ then $0 = Fx \cap H = Fx \cap K$, so again $Fx$ avoids $H/K$. If $x \in H \setminus K$ then $Fx$ does not avoid $H/K$, and $Fx$ covers $H/K$ if and only if $H = K + Fx$, whence the result.
\end{proof}

\begin{coro} Every one-dimensional subalgebra of $L$ is a $CAP$-subalgebra of $L$ if and only if $L$ is supersolvable.
\end{coro}

\begin{propo} If $L$ is supersolvable then every subalgebra of $L$ is a $CAP$-subalgebra.
\end{propo}
\begin{proof} Let $U$ be a subalgebra of $L$ and let $H/K$ be a chief factor of $L$. Suppose first that $U \cap H \subseteq K$. Then $U \cap H \subseteq U \cap K \subseteq U \cap H$, whence $U \cap H = U \cap K$. So suppose now that  $U \cap H \not \subseteq K$. Then,  since $\dim (H/K) =1$, we have that $H = K+ U \cap H$, whence $H+U = K+U$.      
\end{proof}

\section{Some characterisations of solvable algebras}
In this section we are seeking characterisations of solvable Lie algebras in terms of $CAP$-subalgebras. The results are analogues of those for groups as obtained in \cite[Section 3]{xs}, but the proofs are different. A subalgebra $U$ of a Lie algebra $L$ is called a {\em c-ideal} of $L$ if there is an ideal $C$ of $L$ such that $L = U + C$ and $U \cap C \leq U_L$; c-ideals were introduced in \cite{cideal}.  First we need the following result.

\begin{propo}\label{p:solv}
Let $L$ be a Lie algebra over a field $F$ which has characteristic zero, or is algebraically closed and of characteristic greater than $5$, with minimal ideal $A$ and maximal subalgebra $M$. If $M$ is solvable and $M \cap A = 0$ then $L$ is solvable.
\end{propo}
\begin{proof} Clearly $L = M \oplus A$. But now $M$ is a c-ideal of $L$ and it follows from \cite[Theorems 3.2 and 3.3]{cideal} that $L$ is solvable, a contradiction again. 
\end{proof}

\begin{coro}\label{c:solv}  Let $L$ be a Lie algebra over a field $F$ which has characteristic zero, or is algebraically closed field and of characteristic greater than $5$. Then $L$ is solvable if and only if there is a maximal subalgebra $M$ of $L$ such that $M$ is a solvable $CAP$-subalgebra of $L$.
\end{coro}
\begin{proof} If $L$ is solvable it is easy to see that every maximal subalgebra of $L$ is a $CAP$-subalgebra of $L$. So suppose now that $L$ is the smallest non-solvable Lie algebra which has a solvable maximal subalgebra $M$ that is a $CAP$-subalgebra of $L$. If $M_L \neq 0$ then $L/M_L$ must be solvable, whence $L$ is solvable, a contradiction. Hence $M_L = 0$. Now $L$ cannot be simple, by Lemma \ref{l:pre} (iv), so let $A$ be a minimal ideal of $L$ with $A \not \subseteq M$. Since $M$ is a $CAP$-subalgebra we have $M \cap A = 0$. But then $L$ is solvable, by Proposition \ref{p:solv}, a contradiction.
\end{proof}
\bigskip

The Lie algebra $L$ is called {\em monolithic} with {\em monolith} $A$ if $A$ is the unique minimal ideal of $L$. We denote by $\phi(L)$ the Frattini ideal of $L$. If all of the maximal subalgebras of $L$ are $CAP$-subalgebras of $L$ we can deduce solvability without any restrictions on the field $F$.

\begin{theor}\label{t:max}  Let $L$ be a Lie algebra over any field $F$. Then $L$ is solvable if and only if all of its maximal subalgebras are $CAP$-subalgebras.
\end{theor}
\begin{proof}  If $L$ is solvable it is easy to see that every maximal subalgebra of $L$ is a $CAP$-subalgebra of $L$. So suppose that $L$ is the smallest non-solvable Lie algebra all of whose maximal subalgebras are $CAP$-subalgebras. Then $L$ is not simple, by Lemma \ref{l:pre} (iv), so let $A$ be a minimal ideal of $L$. By the minimality of $L$, $L/A$ is solvable. If $L$ has two different minimal ideals $A_1$ and $A_2$, then $L/A_1$, $L/A_2$ and hence $L \cong L/(A_1 \cap A_2)$ is solvable. It follows that $L$ is monolithic with monolith $A$.
\par

Let $M$ be any maximal subalgebra of $L$. Since $M$ is a $CAP$-subalgebra of $L$ we have that either $M+A =M$, whence $A \subseteq M$, or $M \cap A = 0$. If the former holds for every maximal subalgebra $M$, then $A \subseteq \phi(L)$, whence $A$ is abelian and $L$ is solvable. Thus, the latter must hold for some maximal subalgebra $K$. But, for any such maximal subalgebra $K$, $L = K \oplus A$ and $K \cong L/A$ is a solvable c-ideal of $L$. Moreover, if $M$ is a maximal subalgebra of $L$ with $A \subseteq M$, then $M/A$ is a maximal subalgebra of $L/A$ and so is a c-ideal of $L/A$, by \cite[Theorem 3.1]{cideal}. It follows that $M$ is a c-ideal of $L$, by \cite[Lemma 2.1]{cideal}. Hence $L$ is solvable, by \cite[Theorem 3.1]{cideal}. This contradiction establishes the result. 
\end{proof}
\bigskip

Let $M$ be a maximal subalgebra of $L$ and let $K$ be a maximal subalgebra of $M$. Then we call $K$ a {\em $2$-maximal} subalgebra of $L$. Next we consider Lie algebras in which every $2$-maximal subalgebra is a $CAP$-subalgebra of $L$. If $x \in L$ we put $C_L(x) = \{ y \in L : [y,x] = 0 \}$, the {\em centraliser} of $x$ in $L$. We say that $L$ has the {\em one-and-a-half generation property} if, given any $x \in L$, there exists $y \in L$ such that the subalgebra generated by $x$ and $y$, $\langle x, y \rangle$, is $L$. First we need the following result concerning simple Lie algebras with a one-dimensional maximal subalgebra.

\begin{theor}\label{t:gen} Let $L$ be a simple Lie algebra over a perfect field $F$ of characteristic zero or $p>3$. Then $L$ has a one-dimensional maximal subalgebra if and only if $L$ is three-dimensional simple and $\sqrt{F} \not \subseteq F$. 
\end{theor}
\begin{proof} Suppose that $L$ has a one-dimensional maximal subalgebra $Fx$. Clearly $L$ has rank one and $Fx$ is a Cartan subalgebra of $L$. Let $\Gamma$ denote the centroid of $L$. Since $\Gamma x$ is an abelian subalgebra of $L$, we have that $\Gamma x < C_L(x) = Fx$. So $\Gamma = F$, and $L$ is central-simple. Suppose that $\dim L > 3$. It follows from \cite{bo} that $L$ is a form of an Albert-Zassenhaus algebra. Moreover, $L$ has the one-and-a-half generation property. For, given any $y \in L$, either $y=\alpha x$ for some $\alpha \in F$, in which case $\langle y, z \rangle = L$ for any $z \not \in Fx$, or else $y \not \in Fx$, and then $\langle y,x \rangle = L$. Thus, $L$ is a form of a Zassenhaus algebra, by \cite{bois}.
\par

Let $K$ be a splitting field for the minimal polynomial of ad\,$x$ over $F$, and let $G$ be the Galois group of $K$ over $F$. Let $\sigma \in G$. Then $\sigma' = 1 \otimes \sigma$ is a Lie automorphism of $L \otimes_F K = L_K$. As $K$ is a Galois extension of $F$, an element of $L_K$ lies in $L$ if and only if it is fixed by $\sigma'$ for every
$\sigma \in G$. Now $L_K$ has a unique maximal subalgebra $M$ containing $Kx$ of codimension one in $L_K$ and $\sigma'$ must fix $M$. It follows that $(M\cap L)_K=M$ (see \cite[p. 54]{borel}) and so $M\cap L$ is a subalgebra of $L$ of codimension one in $L$. We must have $M\cap L=Fx$, which is impossible. Hence $L$ is three-dimensional simple and, as is well known, has a one-dimensional maximal subalgebra if and only if $\sqrt{F} \not \subseteq F$.
\par

The converse is easy.
\end{proof}

\begin{theor}  Let $L$ be a Lie algebra over any field $F$, in which every $2$-maximal subalgebra of $L$ is a $CAP$-subalgebra. Then either
\begin{itemize}
\item[(i)] $L$ is solvable, or 
\item[(ii)] $L$ is simple and every maximal subalgebra of $L$ is one-dimensional; in particular, if $F$ is perfect and of characteristic zero or $p > 3$, $L$ is three-dimensional simple and $\sqrt{F} \not \subseteq F$. 
\end{itemize}
\end{theor}
\begin{proof} Suppose first that $L$ is simple. Then every $2$-maximal is $0$ and so every maximal subalgebra of $L$ is one-dimensional, which is case (ii). So let $A$ be a minimal ideal of $L$. Suppose first that $A$ is a maximal subalgebra of $L$. Then every maximal subalgebra of $A$ is a $2$-maximal subalgebra of $L$ and so is a $CAP$-subalgebra of $L$. It follows that every maximal subalgebra of $A$ is $0$ and hence that $\dim A=1$. Also, by the maximality of $A$, $\dim (L/A) = 1$ and $L$ is solvable. 
\par

So now assume that $A$ is not a maximal subalgebra of $L$ and that $L$ is a minimal counter-example. Suppose first that $L/A$ is as in (ii). Let $Fx+A$ be a maximal subalgebra of $L$ and let $K$ be a $2$-maximal subalgebra of $L$ with $Fx \subseteq K \subset Fx+A$. Clearly $A \not \subseteq K$, so $K \cap A=0$, since $K$ is a $CAP$-subalgebra of $L$. Now $L/A$ is a chief factor of $L$ and $K \neq 0$, so $L=K \oplus A=Fx+A$, a contradiction.
\par

Thus $L/A$ is solvable and $L$ is monolithic, as in Theorem \ref{t:max}. If $A \subseteq \phi(L)$ then $A$ is solvable and hence so is $L$. Thus, $\phi(L) = 0$ and $L = M+A$ for some maximal subalgebra $M$ of $L$. Suppose that $M \cap A \neq 0$. Let $K$ be a maximal subalgebra of $M$ with $M \cap A \subseteq K$. Then $K$ is a $2$-maximal subalgebra of $L$ and so either $K+A=K$, yielding $A \subseteq K \subseteq M$, or $M \cap A \subseteq K \cap A =0$, both of which are contradictions. It follows that $M \cong L/A$ is a solvable c-ideal, as is any maximal subalgebra of $L$ not containing $A$. But every maximal subalgebra containing $A$ is a c-ideal, as in Theorem \ref{t:max}, and the result follows similarly.
\end{proof}
\bigskip

\begin{ex}Note that there are solvable Lie algebras with $2$-maximal subalgebras which are not $CAP$-subalgebras. For example, let $L = \R e_1 + \ Re_2 + \R e_3$ with $[e_1,e_3] = -[e_3,e_1] = e_2$, $[e_2,e_3]=-[e_3,e_2]=-e_1$ and all other products zero. Then $A = \R e_1 + \R e_2$ is a minimal ideal of $L$ and $U = \R e_1$ is a 2-maximal subalgebra of $L$. However, $A+U=A \neq U=0+U$ and $A \cap U=U \neq 0=0 \cap U$, so $U$ is not a $CAP$-subalgebra of $L$.  
\end{ex}

\begin{lemma}\label{l:2-max}  Let $L$ be a solvable Lie algebra. Then there is a $2$-maximal subalgebra $K$ of $L$ which is an ideal of $L$, and hence a $CAP$-subalgebra of $L$.
\end{lemma}
\begin{proof} If $\dim (L/L^2) >1$ there is clearly a $2$-maximal subalgebra of $L$ containing $L^2$, so suppose that $\dim (L/L^2) =1$. Let $L = L^2 + Fx$, and let $L^2/K$ be a chief factor of $L$. Suppose that $K+Fx \subset U$, where $U$ is a subalgebra of $L$. Then $[U \cap L^2,L] = [U \cap L^2,L^2] + [U \cap L^2,Fx] \subseteq U \cap L^2$, since $L^{(2)} = [L^2,L^2] \subseteq K$. It follows that $L^2 \subseteq U \cap L^2$, whence $K+Fx$ is a maximal subalgebra and $K$ a $2$-maximal subalgebra of $L$.
\end{proof}
\bigskip

Finally we seek to characterise Lie algebras having a solvable $2$-maximal subalgebra which is a $CAP$-subalgebra of $L$.

\begin{theor} Let $L$ be a Lie algebra over a field $F$ which has characteristic zero. Then $L$  has a solvable $2$-maximal subalgebra $K$ of $L$ that is a $CAP$-subalgebra of $L$ if and only if either
\begin{itemize}
\item[(i)] $L$ is solvable, or 
\item[(ii)] $L=R \oplus S$, where $R$ is the (solvable) radical of $L$ (possibly $0$), $S$ is three-dimensional simple and  $\sqrt{F} \not \subseteq F$.
\end{itemize}
\end{theor}
\begin{proof} Suppose that $K$ is a solvable $2$-maximal subalgebra of $L$ that is a $CAP$-subalgebra of $L$, and that $R$ is the radical of $L$. Then $R+K$ is a solvable subalgebra of $L$. If $R+K=L$ we have case (i). So suppose that $R+K \neq L$. Let $L=R\oplus S$ where $S=S_1 \oplus \ldots \oplus S_n$, $S_i$ is a simple ideal of $S$ and put $J_i=R+S_1 \oplus \ldots \oplus S_i$ for $i=0, \ldots, n$ (where $J_0=R$). Suppose that $K \subseteq J_i$. Since $J_i/J_{i-1}$ is a chief factor of $L$ we have that $J_i=K+J_i=K+J_{i-1}$ or $K=K \cap J_i=K \cap J_{i-1}$. The former implies that $J_i/J_{i-1} \cong K/K \cap J_{i-1}$, which is impossible as $J_i/J_{i-1}$ is simple and $ K/K \cap J_{i-1}$ is solvable. It follows that $K \subseteq J_{i-1}$, from which $K \subseteq R$, since $K \subseteq J_n$.
\par

Let $M$ be a maximal subalgebra of $L$ containing $K$ as a maximal subalgebra. Suppose that $R \not \subset M$, so that $L=R+M$. Then $K \subseteq M \cap R \subseteq M$, so either $M \cap R=M$ or $M \cap R=K$. The former implies that $M \subseteq R$, which is impossible; the latter is also impossible, since $S \cong L/R \cong M/M \cap R$ and $M \cap R$ is not maximal in $M$. Hence $K \subseteq R \subset M$. It follows that $K=R$, from which (ii) easily follows.
\par

It is easy to see that algebras as in (i) and (ii) have a solvable $2$-maximal subalgebra which is a $CAP$-subalgebra.
\end{proof}

\end{document}